\documentclass[10pt]{amsart}
\usepackage{latexsym,amssymb,amsmath}
\usepackage{amsmath,amsfonts}
\usepackage{epsfig}
\usepackage{graphicx}
\usepackage{verbatim,xcolor}
\usepackage{dirtytalk}
\usepackage{hyperref}

\newcommand{\norm}[1]{\left\lVert#1\right\rVert}
\theoremstyle{plain}
\newtheorem{theorem}{Theorem}[section]

\newtheorem{lemma}[theorem]{Lemma}

\newtheorem{proposition}[theorem]{Proposition}

\theoremstyle{remark}
\newtheorem{Remark}[theorem]{Remark}

\allowdisplaybreaks

\newcommand{\Hmm}[1]{\leavevmode{\marginpar{\tiny%
$\hbox to 0mm{\hspace*{-0.5mm}$\leftarrow$\hss}%
\vcenter{\vrule depth 0.1mm height 0.1mm width \the\marginparwidth}%
\hbox to 0mm{\hss$\rightarrow$\hspace*{-0.5mm}}$\\\relax\raggedright #1}}}

\usepackage{scalerel,stackengine}
\stackMath
\newcommand\reallywidehat[1]{%
\savestack{\tmpbox}{\stretchto{%
  \scaleto{%
    \scalerel*[\widthof{\ensuremath{#1}}]{\kern-.6pt\bigwedge\kern-.6pt}%
    {\rule[-\textheight/2]{1ex}{\textheight}}
  }{\textheight}%
}{0.5ex}}%
\stackon[1pt]{#1}{\tmpbox}%
}

\numberwithin{equation}{section}

\begin{document}


\author[Ba\c{s}ako\u{g}lu, Ye\c{s}ilo\u{g}lu, Yılmaz] {Engin Ba\c{s}ako\u{g}lu, Barı\c{s} Ye\c{s}ilo\u{g}lu, O\u{g}uz Yılmaz}

\address{School of Mathematics, University of Birmingham, Watson Building, Edgbaston, Birmingham B15 2TT, United Kingdom}
\email{e.basakoglu@bham.ac.uk}
\address{Department of Mathematics,
Bo\u gazi\c ci University, 
Bebek 34342, Istanbul, Turkey}
\email{baris.yesiloglu@boun.edu.tr}
\address{Department of Mathematics, Bo\u gazi\c ci University, 
Bebek 34342, Istanbul, Turkey}
\email{oguz.yilmaz@boun.edu.tr}

\begin{abstract}
We study the global well-posedness of the two-dimensional defocusing fourth-order Schrödinger initial value problem with power type nonlinearities $\vert u\vert^{2k}u$ where $k$ is a positive integer. By using the $I$-method, we prove that global well-posedness is satisfied in the Sobolev spaces $H^{s}(\mathbb{R}^{2})$ for $2-\frac{3}{4k}<s<2$.
\end{abstract}

 \title[Global Solutions of 4-NLS Equations below the Energy Space]{Global Well-posedness for The Fourth-order Nonlinear Schr\"{o}dinger Equations on $\mathbb{R}^{2}$}
\maketitle

\section{Introduction}
In this paper, we consider the initial value problem (IVP) for the defocusing fourth-order Schrödinger equation (4-NLS) with the power type nonlinearity
\begin{equation}\label{4NLS}
    \begin{cases}
        i\partial_{t}u+\Delta^{2}u-\Delta u+\vert u\vert^{2k}u=0,\quad (t,x)\in\mathbb{R}\times\mathbb{R}^{2},\\
        u(0,x)=u_{0}(x)\in H^{s}(\mathbb{R}^{2})
    \end{cases}
\end{equation}
where $0<s<2$, $k$ is a positive integer and $u$ is a complex-valued space-time function. The fourth-order nonlinear Schrödinger equations 
\begin{equation*}\label{eq:karpman_4NLS}
    i\partial_{t}\psi+\frac{1}{2}\triangle\psi+\frac{\gamma}{2}\triangle^{2}\psi+f(\vert\psi\vert^{2})\psi=0,\quad f(u)=u^{p},\,p\geq 1,\,\gamma\in\mathbb{R}
\end{equation*}
were introduced by Karpman and Shagalov, \cite{karpman1996stabilization,karpman2000stability}, regarding the effect of small fourth-order dispersion terms in the propagation of intense laser beams in a bulk medium with Kerr nonlinearity.

The local solutions of \eqref{4NLS} have the following mass and energy conservation laws:
\begin{equation}\label{eq:mass_cons}
    M(u(t)):=\int_{\mathbb{R}^{2}}\vert u(t,x)\vert^{2} \text{d}x=M(u_{0}),
\end{equation}
\begin{equation}\label{eq:enrgy_cons}
    E(u(t)):=\frac{1}{2}\int_{\mathbb{R}^{2}}\vert\Delta u(t,x)\vert^{2}+\vert\nabla u(t,x)\vert^{2}+\frac{1}{k+1}\vert u(t,x)\vert^{2k+2}\,\text{d}x=E(u_{0}).
\end{equation}
Note that the IVP \eqref{4NLS} does not possess a scaling symmetry (unlike the 4-NLS without Laplacian), this removes the need to consider the $H^{s}_{x}$-criticality concerns. For instance, if the mass supercriticality for \eqref{4NLS} could be possible, then the mass conservation of solutions would not hold at all.

In order to discuss some of the respective literature briefly, we shall state the general form of the 4-NLS equations as follows 
\begin{equation}\label{eq:gen_form_4NLS}
    i\partial_{t}u+\alpha\Delta^{2}u+\beta\Delta u+\mathcal{N}(u)=0
\end{equation}
where $\alpha,\beta\in\mathbb{R}$, and $\mathcal{N}(u)$ is the nonlinear perturbation. 

The local/global well-posedness as well as the scattering of \eqref{eq:gen_form_4NLS} have been extensively studied. The cubic defocusing biharmonic NLS equation (\eqref{eq:gen_form_4NLS} with $\alpha=1$, $\beta=0$ and $\mathcal{N}(u)=\vert u\vert^{2}u$) in arbitrary space dimensions was investigated by Pausader in \cite{pausader2009cubic}. He established that for $1\leq n\leq 8$, the IVP is globally well-posed for arbitrary $H^{2}(\mathbb{R}^{n})$ initial data, while the equation is ill-posed in the energy space for $n\geq 9$ by showing for any $\varepsilon>0$ the existence of a Schwartz class data $u_{0}$ and a finite time $t_{\varepsilon}\in (0,\varepsilon)$ such that even though $\Vert u_{0}\Vert_{H^{2}_{x}}<\varepsilon$ one has $\Vert u(t_{\varepsilon})\Vert_{H^{2}_{x}}>\varepsilon^{-1}$. He also proved that the scattering in $H^{2}(\mathbb{R}^{n})$ holds true when $5\leq n\leq 8$ and the scattering operator is analytic. In \cite{pausader_mass_crit_2010}, Pausader and Shao addressed the equation \eqref{eq:gen_form_4NLS} with $\alpha=1$, $\beta=0$ and the mass-critical nonlinearity $\mathcal{N}(u)=\pm\vert u\vert^{\frac{8}{n}}u$. They established that for $n\geq 5$ the defocusing equation is globally well-posed and scatters in time for arbitrary $L^{2}$ initial data. Moreover, under the additional assumption that the mass of the initial data is strictly less than the mass of the ground state solution, they showed that the focusing equation is also globally well-posed and scatters in time for radially symmetric $L^{2}$-initial data. The mass supercritical equation \eqref{eq:gen_form_4NLS} with $\alpha=1$, $\beta=0,1$ and nonlinearities $\mathcal{N}(u)=\vert u\vert^{p-1}u$, $1+\frac{8}{n}<p<\infty$ for $1\leq n\leq 4$ was studied by Pausader and Xia in \cite{pausader2013scattering}, who proved that the equation scatters in time in the energy space $H^{2}$ by utilizing a new virial-type estimate. For additional results of global well-posedness and scattering in time below the energy space for the biharmonic NLS, see for instance \cite{dinh2017global}, \cite{seong2021well}. As regards to more general 4-NLS, Guo and Cui \cite{GUO2006706} considered the IVP for \eqref{eq:gen_form_4NLS} with $\alpha\neq 0$, $\beta\in\mathbb{R}$, and nonlinearity $\mathcal{N}(u)=f(\vert u\vert^{2})u$, for a real-valued nonlinear function $f$, in dimensions $n=1,2,3$. With the smoothness condition $f\in C^{k+1}(\mathbb{R}_{+};\mathbb{R})$ for $k\geq 2$ being an integer, they established that for any $u
_{0}\in H^{k}(\mathbb{R}^{n})$ initial data, the equation has a unique global solution $u$ in
$C(\mathbb{R};H^{k}(\mathbb{R}^{n}))\cap C^{1}(\mathbb{R};H^{k-2}(\mathbb{R}^{n})).$ Concerning the equation \eqref{eq:gen_form_4NLS} with $\alpha=1$, $\beta=-1,0,1$, by using the $[k,Z]$ multiplier method, Zheng \cite{Zheng_2011} obtained the $H^{s}(\mathbb{R})$ local well-posedness for $s\in(-\frac{7}{4},0]$ if $\mathcal{N}(u)=u^2, \overline{u}^{2}$; and for $s\in(-\frac{3}{4},0]$ if $\mathcal{N}(u)=|u|^2$. The well-posedness of \eqref{eq:gen_form_4NLS} in modulation spaces were studied in \cite{HAO200758}, \cite{RUZHANSKY201631}. In \cite{Liu_2021}, Liu and Zhang established the local well-posedness of the IVP \eqref{eq:gen_form_4NLS} in $H^{s}(\mathbb{R}^{n})$ for $n\geq 1$, with $\alpha=1$, $\beta=\pm 1$ or $0$ and $\mathcal{N}(u)=\lambda\vert u\vert^{p}u$ where $\lambda\in\mathbb{C}$ and $0<p<\frac{8}{n-2s}$. With this result, they have extended the pre-established local well-posedness results in the energy space to the full range of energy subcritical powers.

In this paper, we aim to study the global well-posedness for the IVP \eqref{4NLS} below the energy space. Our result reads as follows:
\begin{theorem}\label{eq:main_thm}
    The initial value problem \eqref{4NLS} is globally well-posed for the initial data $u_{0}\in H^{s}(\mathbb{R}^{2})$ for $s>2-\frac{3}{4k}$.
\end{theorem}
\begin{Remark}
In \cite{GUO2010}, Guo established a global well-posedness result in $H^{s}(\mathbb{R}^{n})$ for the equation \eqref{eq:gen_form_4NLS} with $\mathcal{N}(u)=\nu\vert u\vert^{2m}u$ for $s>1 + \frac{mn - 9 + \sqrt{(4m-mn+7)^2 +16}}{4m}$, $m\in\mathbb{N}$ such that $4<mn<4m+2$ (with either $\alpha>0$, $\beta<0$, $\nu>0$ or $\alpha<0$, $\beta>0$, $\nu<0$; hence the resulting equations are defocusing in either case). In this paper, we improve the result in \cite{GUO2010} for $n=2$ and $k>2$ integer, besides unlike \cite{GUO2010} our result covers the cubic and quintic 4-NLS equations. Note that, in view of the absence of scaling symmetry, the exponent of the upper bound of the local existence time is of utmost importance in obtaining better results. This exponent is given as $-\frac{4}{s-s_{c}}$ with $s_{c}=\frac{n}{2}-\frac{2}{m}$ in Lemma $2.3$ in \cite{GUO2010}, whereas the respective exponent in our discussion is given in Lemma \ref{mod_local_exst}. In order to obtain the result, Guo utilizes the Strichartz spaces $L^{q}_{t}L^{r}_{x}$ with the biharmonic admissible pair $(q,r)$ \eqref{eq:bih_admissible}; essentially, the argument in \cite{GUO2010} relies heavily on the $I$-method implemented on such spaces. Our improvement relies on the use of $X^{s,b}$ spaces together with the derivative gain (see \eqref{keyestimate}) due to the strong dispersive nature of 4-NLS, which was first exploited by Ben Artzi, Koch, and Saut \cite{BENARTZI200087}. Indeed, the same argument could be used to improve the result in \cite{GUO2010} in other dimensions as well.    
\end{Remark}
\begin{Remark}
    The identity \eqref{eq:exponent_of_time_param} in Lemma \ref{mod_local_exst} will play a crucial role in estimating the difference $E(Iu(\delta))-E(Iu(0))$ in the lack of scale invariance. In particular, the bigger the exponent of $\delta$ in the estimate \eqref{eq:exponent_of_time_param} we have, the better the range of global well-posedness one can establish (see \eqref{eq:T_delta_N_relation} below). In a recent paper \cite{başakoğlu2023global}, G\"urel, the first and the third authors have established global well-posedness of the quintic biharmonic NLS in $H^s(\mathbb{R}^2)$ for $\frac{8}{7}<s<2$ by utilizing the scale invariance property of the biharmonic equation. Compared to the result of the quintic 4-NLS ($k=2$ case) in this paper, the result of the paper \cite{başakoğlu2023global} is stronger, which is the indication of how scaling property of a particular equation yields better range of global well-posedness.  Besides the failure of a scaling invariance, another failure of the 4-NLS equation \eqref{4NLS} is the lack of Galilean symmetry which causes a failure of having a reference frame that the momentum of the solution $u$ is $0$ (which is called normalization of $u$) see \cite{pausader_mass_crit_2010}, \cite{pausader2013scattering}. Unlike the fourth order nonlinear Schrödinger equations, nonlinear Schrödinger equations satisfy both scaling and Galilean symmetries. For instance, in \cite{dodson_2d_nls}, Dodson exploited both scaling and Galilean symmetries of the mass-critical NLS to prove the global well-posedness and scattering in time for $L^2(\mathbb{R}^2)$ initial data by means of a concentration compactness argument, see theorems 1.3, 1.5 and 1.6 in \cite{dodson_2d_nls}. As for the biharmonic equation, in \cite{pausader_mass_crit_2010}, Pausader and Shao utilized a similar approach as Dodson used in \cite{dodson_2d_nls} (not the Galilean symmetry, which is not available for biharmonic NLS, but the symmetry group generated by temporal-spatial translation and scaling of the biharmonic equation for concentration compactness approach to be applied) to establish the global existence and scattering of defocusing and focusing mass-critical biharmonic equation with $L^{2}(\mathbb{R}^d)$ data, for $d\geq 5$. Also in the lower dimensions, as mentioned above, Pausader and Xia \cite{pausader2013scattering} utilized a profile decomposition and concentration compactness approach together with a new virial-type inequality to prove scattering in the energy space $H^2$. To establish a better global existence result for \eqref{4NLS} below the energy space though, one needs to take advantage of a more refined approach that works without the use of symmetries discussed above and this will be taken into consideration in a further study. 
\end{Remark}
Now we outline the organization of the paper. In Section 2, we present notations and the necessary background, including a priori estimates. In Section 3, we prove an almost conservation law and a local existence result for the modified equation \eqref{eq:mod_4NLS}. Lastly, in Section 4 we give a proof of Theorem \ref{eq:main_thm} by using the results obtained in Section 3.
\section{Background and Notation}
We write $X\lesssim Y$ if there exists an absolute constant $C>0$ such that $X\leq CY$, also define $X\sim Y$ when $X\lesssim Y$ and $Y\lesssim X$. We use the notation $X\ll Y$ meaning that $X\leq CY$ where $C$ is a sufficiently large constant. The Fourier transform is defined as
\begin{equation*}
    \widehat{f}(\xi)=\int_{\mathbb{R}^{d}}e^{-ix\cdot\xi}f(x)\,\text{d}x,
\end{equation*}
 similarly the space-time Fourier transform is given as \begin{equation*}\widetilde{f}(\tau,\xi)=\int_{\mathbb{R}^{d+1}}e^{-ix\cdot\xi-it\tau}f(t,x)\,\text{d}x\text{d}t.\end{equation*}For $s\in\mathbb{R}$, the Sobolev space $H^s(\mathbb{R}^d)$ is defined via the norm \begin{equation*}
  \Vert f\Vert_{{H}^{s}(\mathbb{R}^{d})}=\Vert\langle\xi\rangle^{s}\widehat{f}(\xi)\Vert_{L^{2}_{\xi}(\mathbb{R}^{d})}
 \end{equation*}
 where $\langle \cdot\rangle=\sqrt{1+|\cdot|^2}$. Also the homogeneous Sobolev space $\dot{H}^s(\mathbb{R}^d)$ is defined similarly by using $|\cdot|$ rather than $\langle \cdot\rangle$. The Fourier restriction space $X^{s,b}$ associated with the equation \eqref{4NLS} is defined to be the closure of the Schwartz functions space $\mathcal{S}_{t,x}(\mathbb{R}\times \mathbb{R}^d)$ under the norm
\begin{equation}\label{eq:X_sb}
    \Vert u\Vert_{X^{s,b}(\mathbb{R}\times \mathbb{R}^d)}=\Vert\langle\xi\rangle^{s}\langle\tau-\vert\xi\vert^{4}-\vert\xi\vert^{2}\rangle^{b}\widetilde{u}(\tau,\xi)\Vert_{L^{2}_{\tau,\xi}(\mathbb{R}\times\mathbb{R}^d)},
\end{equation}
and we also define the restricted $X^{s,b}$ space, $X^{s,b}_{\delta}$, as the equivalence classes of functions that agree on $[0,\delta]$ with the norm 
\begin{equation}\label{eq:trunc-X_sb}
    \Vert u\Vert_{X^{s,b}_{\delta}}=\inf_{v=u\text{ on }[0,\delta]}\Vert v\Vert_{X^{s,b}}.
\end{equation}
Let $\varphi$ be a real-valued, smooth, compactly supported, radial function such that $\text{supp}\,\varphi\subseteq\{\xi\in\mathbb{R}^{n}:\vert\xi\vert\leq 2\}$ and $\varphi\equiv 1$ on the closed unit ball. For a dyadic number $N$, the Littlewood-Paley projection operators are defined by
\begin{align}
    \widehat{P_{N}u}(\xi)&=(\varphi(\xi/N)-\varphi(2\xi/N))\widehat{u}(\xi)\label{eq:L-P_opt}\\ \nonumber\widehat{P_{\leq N}u}(\xi)&=\sum_{M\leq N}\widehat{P_{M}u}(\xi)=\varphi(\xi/N)\widehat{u}(\xi),\\\nonumber \widehat{P_{>N}u}(\xi)&=\sum_{M>N}\widehat{P_{M}u}(\xi)=(1-\varphi(\xi/N))\widehat{u}(\xi)
\end{align}
where $M$ ranges over dyadic numbers. In our discussion the following Bernstein's estimates are useful: for $s\geq 0$ and $1\leq p\leq\infty$,
\begin{align}
    \Vert P_{N}u\Vert_{L^{p}_{x}(\mathbb{R}^{n})}\sim_{p,n,s}N^{\pm s}\Vert\vert\nabla\vert^{\mp s}P_{N}u\Vert_{L^{p}_{x}
    (\mathbb{R}^{n})}\label{eq:Berns_ineq}\\
    \norm{P_{\geq N}f}_{L_x^{p}(\mathbb{R}^n)}\lesssim_{p,n,s} N^{-s}\norm{|\nabla|^sP_{\geq N}f}_{L_x^{p}(\mathbb{R}^n)}.\label{eq:Berns_ineq2}
\end{align}
The Sobolev spaces can be characterized by using the Littlewood–Paley projections as follows
\begin{equation*}
    \Vert u\Vert_{H^{s}_{x}(\mathbb{R}^{n})}\sim_{s,n}\Vert P_{\leq 1}u\Vert_{L^{2}_{x}(\mathbb{R}^{n})}+\Big(\sum_{N>1}N^{2s}\Vert P_{N}u\Vert^2_{L^{2}_{x}(\mathbb{R}^{n})}\Big)^{1/2}
\end{equation*}
and
\begin{equation*}
    \Vert u\Vert_{\Dot{H}^{s}_{x}(\mathbb{R}^{n})}\sim_{s,n}\Big(\sum_{N}N^{2s}\Vert P_{N}u\Vert^2_{L^{2}_{x}(\mathbb{R}^{n})}\Big)^{1/2}.
\end{equation*} We define the $L^{q}_{t}L^{p}_x$ norm by
\begin{equation*}\label{eq:LpLqnorm}
    \Vert f\Vert_{L^{q}_{t}L^{p}_{x}(\mathbb{R}\times\mathbb{R}^{d})}=\left(\int_{\mathbb{R}}\left(\int_{\mathbb{R}^{d}}\vert f(t,x)\vert^{p}\,\text{d}x\right)^{q/p}\text{d}t\right)^{1/q}
\end{equation*}
and the pair $(p,q)$ is said to be biharmonic admissible if 
\begin{equation} \label{eq:bih_admissible}
    \frac{4}{q}=d\Big(\frac{1}{2}-\frac{1}{p}\Big)\,\,\text{and}\,\, \begin{cases}
2\leq p < \frac{2d}{d-4} \hspace{0.5cm}\text{if}\,\,d\geq 4\\ 2\leq p \leq \infty \hspace{0.75cm}\text{if}\,\,d<4.
\end{cases}
\end{equation}
Note that $(p,q,d)\neq (\infty,2,4)$.
Below we state a very useful refinement of the Strichartz estimates which allows us to gain some order of derivative. For more general estimates of this type for a class of dispersive equations, see \cite{cui2007well}.
\begin{lemma}\label{dergain}
Assume $0\leq\mu\leq 1$, $\frac{2}{1-\mu}\leq p\leq\infty$, $2\leq q\leq\infty$, and
\begin{equation}\label{eq:mu_admiss_pairs}
    \frac{4}{q}=2\Big(\frac{1}{2}-\frac{1}{p}\Big)+\mu.
\end{equation}
Then for any $T_{0}>0$, there is a constant $C>0$ depending on $p,q,T_{0},\mu$ such that for any $0<T<T_{0}$ and any $f\in L^{2}(\mathbb{R}^{2})$, we have
\begin{align*}
    \Vert e^{it(\Delta^{2}-\Delta)}\vert\nabla\vert^{\mu}f\Vert_{L^{q}_{t}L^{p}_{x}([0,T]\times\mathbb{R}^{2})}&\leq C\Vert f\Vert_{L^{2}_{x}(\mathbb{R}^{2})}\\
    \Vert e^{it(\Delta^{2}-\Delta)}\langle\nabla\rangle^{\mu}f\Vert_{L^{q}_{t}L^{p}_{x}([0,T]\times\mathbb{R}^{2})}&\leq C\Vert f\Vert_{L^{2}_{x}(\mathbb{R}^{2})}.
\end{align*}
\end{lemma}
\noindent By making use of Lemma \ref{dergain}, we might gain certain amount of derivative as the following lemma shows. 
\begin{lemma}\label{oguz}
    Let $u\in X^{0,\frac{1}{2}+}$ with the spatial frequency support in the set $\{\xi\in\mathbb{R}^{2}:\vert\xi\vert\sim N\}$ for $N\geq 1$. Then for sufficiently small $\delta>0$, we have
    \begin{equation}\label{eq:refined_str_est}
        \Vert\vert\nabla\vert^{\mu}u\Vert_{L^{q}_{t}L^{p}_{x}([0,\delta]\times\mathbb{R}^{2})}\lesssim\Vert u\Vert_{X^{0,\frac{1}{2}+}_{\delta}}
    \end{equation}
    where $\mu,p,q$ are given as in Lemma \ref{dergain}.
\end{lemma}
\begin{proof}
    The proof proceeds as in \cite{başakoğlu2023global} with minor modifications. 
\end{proof}
\begin{Remark}
By means of the derivative gain in \eqref{eq:refined_str_est} together with the Bernstein's inequality \eqref{eq:Berns_ineq}, for $u$ as in Lemma \ref{oguz}, one has
\begin{equation}\label{keyestimate}
    \Vert u\Vert_{L^{q}_{t}L^{p}_{x}([0,\delta]\times\mathbb{R}^{2})}\lesssim N^{-\mu}\Vert u\Vert_{X^{0,\frac{1}{2}+}_{\delta}}.
\end{equation}
\end{Remark}
\section{Almost Conservation Law}
Given $s<2$ and a parameter $N\gg 1$, define the Fourier multiplier operator 
\begin{equation}\label{eq:I-opt}
    \widehat{I_{N}f}(\xi)=m_{N}(\xi)\widehat{f}(\xi),
\end{equation}
where 
\begin{equation}\label{eq:defn-of-m}
    m_{N}(\xi)=\begin{cases}
    1\quad&\text{if}\,\,\vert\xi\vert\leq N,\\
    \vert\xi\vert^{s-2}N^{2-s}\quad &\text{if}\,\,\vert\xi\vert>2N.
    \end{cases}
\end{equation}
Note that $m_N$ is smooth, radial and  non-increasing in $\vert\xi\vert$.
For simplicity, we shall drop the subscript $N$ in \eqref{eq:I-opt} and \eqref{eq:defn-of-m}. Hence, the multiplier $m$ satisfies the condition
\begin{equation*}
    \vert\nabla_{\xi}^{j}m(\xi)\vert\lesssim\vert\xi\vert^{-j}\text{ for $j\geq 0$, and $\xi\in\mathbb{R}^{n}\setminus\{0\}$}
\end{equation*}
implying that $m$ is a Hörmander-Mikhlin multiplier (see \cite{Shamir1966}). The modified energy is given by
\begin{equation*}
    E(Iu(t))=\frac{1}{2}\int_{\mathbb{R}^{2}}\vert\Delta Iu\vert^{2}+\vert\nabla Iu\vert^{2}+\frac{1}{k+1}\vert Iu\vert^{2k+2}\,\text{d}x.
\end{equation*}
By taking the time derivative of the modified energy, using Green's identities, and applying Plancherel's formula, we have
\begin{align*}
    \frac{d}{dt}E(Iu(t))&=\Im\int_{\sum_{i=1}^{2k+2}\xi_{i}=0}\Big(1-\frac{m(\xi_{2\dots 2k+2})}{m(\xi_{2})\dots m(\xi_{2k+2})}\Big)\reallywidehat{\overline{\Delta^{2}Iu}}(\xi_{1})\widehat{Iu}(\xi_{2})\widehat{\overline{Iu}}(\xi_{3})\dots\widehat{Iu}(\xi_{2k+2})\\&+\Im\int_{\sum_{i=1}^{2k+2}\xi_{i}=0}\Big(\frac{m(\xi_{2\dots 2k+2})}{m(\xi_{2})\dots m(\xi_{2k+2})}-1\Big)\reallywidehat{\overline{\Delta Iu}}(\xi_{1})\widehat{Iu}(\xi_{2})\widehat{\overline{Iu}}(\xi_{3})\dots\widehat{Iu}(\xi_{2k+2})\\&+\Im\int_{\sum_{i=1}^{2k+2}\xi_{i}=0}\Big(1-\frac{m(\xi_{2\dots 2k+2})}{m(\xi_{2})\dots m(\xi_{2k+2})}\Big)\reallywidehat{I(\vert u\vert^{2k}\overline{u})}(\xi_{1})\widehat{Iu}(\xi_{2})\widehat{\overline{Iu}}(\xi_{3})\dots\widehat{Iu}(\xi_{2k+2}).
    \end{align*}
    Therefore, in estimating the quantity $E(Iu(\delta))-E(Iu(0))$, we deal with the following terms
    \begin{align*}&\Big|\int_0^{\delta}\int_{\sum_{i=1}^{2k+2}\xi_{i}=0}\Big(1-\frac{m(\xi_{2\dots 2k+2})}{m(\xi_{2})\dots m(\xi_{2k+2})}\Big)\reallywidehat{\overline{\Delta^{2}Iu}}(\xi_{1})\widehat{Iu}(\xi_{2})\widehat{\overline{Iu}}(\xi_{3})\dots\widehat{Iu}(\xi_{2k+2})\Big|\\&+\Big|\int_0^{\delta}\int_{\sum_{i=1}^{2k+2}\xi_{i}=0}\Big(1-\frac{m(\xi_{2\dots 2k+2})}{m(\xi_{2})\dots m(\xi_{2k+2})}\Big)\reallywidehat{\overline{\Delta Iu}}(\xi_{1})\widehat{Iu}(\xi_{2})\widehat{\overline{Iu}}(\xi_{3})\dots\widehat{Iu}(\xi_{2k+2})\Big|\\&+\Big|\int_0^{\delta}\int_{\sum_{i=1}^{2k+2}\xi_{i}=0}\Big(1-\frac{m(\xi_{2\dots 2k+2})}{m(\xi_{2})\dots m(\xi_{2k+2})}\Big)\reallywidehat{I(\vert u\vert^{2k}\overline{u})}(\xi_{1})\widehat{Iu}(\xi_{2})\widehat{\overline{Iu}}(\xi_{3})\dots\widehat{Iu}(\xi_{2k+2})\Big|\\&=: Term_{1}+Term_{2}+Term_{3}
\end{align*}
for which we aim to establish a reasonable bound depending on the spatial dimension, frequency cut-off parameter $N$, $k$, and the $H^{s}$-norm of the initial data $u_{0}$. Before proceeding further to this issue, we need to establish a local existence result for the modified equation
\begin{equation}\label{eq:mod_4NLS}
    \begin{cases}
        i\partial_{t}Iu+\Delta^{2}Iu-\Delta Iu+I(\vert u\vert^{2k}u)=0,\\
        Iu(0,x)=Iu_{0}(x).
    \end{cases}
\end{equation}
\begin{lemma}[Modified Local Existence]\label{mod_local_exst} Assume $2-\frac{3}{4k}<s<2$. Let $u_{0}\in H^{s}_{x}(\mathbb{R}^{2})$ be given. Then there exists a positive number $\delta$ and $c>0$ such that the IVP \eqref{eq:mod_4NLS} has a unique local solution $Iu\in C([0,\delta];H_{x}^{2}(\mathbb{R}^{2}))$ such that
\begin{equation*}
    \Vert Iu\Vert_{X^{2,\frac{1}{2}+}_{\delta}}\leq c\Vert Iu_{0}\Vert_{H^{2}_{x}(\mathbb{R}^{2})}.
\end{equation*}
Moreover, the existence time can be estimated by 
\begin{equation}\label{eq:exponent_of_time_param}
    \delta^{1-}\sim \frac{1}{\Vert Iu_{0}\Vert^{2k}_{H^{2}_{x}(\mathbb{R}^{2})}}.
\end{equation}
\end{lemma}
\begin{proof}
    Let $\eta$ be a compactly supported, smooth, radial function in the time variable satisfying $0\leq\eta\leq 1$, $\eta\equiv 1$ on $[-1,1]$ and $supp(\eta)\subseteq[-2,2]$. Define $\eta_{\delta}(t)=\eta(t/\delta)$ for $\delta>0$. We begin with recalling the standard $X^{s,b}$-estimates:
\begin{align}
    \Vert\eta(t)e^{it(\Delta^{2}-\Delta)} Iu_{0}\Vert_{X^{2,\frac{1}{2}+}}&\lesssim\Vert Iu_{0}\Vert_{H^{2}_{x}(\mathbb{R}^{d})},\label{eq:free_soln_est}\\
    \Big\Vert\eta_{\delta}(t)\int_{0}^{t}e^{i(t-t')(\Delta^{2}-\Delta)}I(\vert u\vert^{2k}u)(\cdot,t')\,\text{d}t'\Big\Vert_{X^{2,\frac{1}{2}+}}&\lesssim\Vert I(\vert u\vert^{2k}u)\Vert_{X_{\delta}^{2,-\frac{1}{2}+}},\label{eq:nonl_part_est}\\
    \Vert I(\vert u\vert^{2k}u)\Vert_{X_{\delta}^{2,-\frac{1}{2}+}}&\lesssim\delta^{\frac{1}{2}-}\Vert I(\vert u\vert^{2k}u)\Vert_{X^{2,0}_{\delta}}.\label{eq:time_param_gain}
\end{align}
In order to implement the Banach fixed point argument, we define the integral equation
\begin{equation}
    \Gamma(Iu)(t)=\eta(t)e^{it(\Delta^{2}-\Delta)}Iu_{0}+i\eta_{\delta}(t)\int_{0}^{t}e^{i(t-t')(\Delta^{2}-\Delta)}I(\vert u\vert^{2k}u)(t')\,\text{d}t'.
\end{equation}
Via \eqref{eq:free_soln_est}--\eqref{eq:time_param_gain}, we have
\begin{align*}
    \Vert\Gamma(Iu)\Vert_{X^{2,\frac{1}{2}+}}\leq c\Vert Iu_{0}\Vert_{H^{2}_{x}}+c\delta^{\frac{1}{2}-}\Vert I(\vert u\vert^{2k}u)\Vert_{X^{2,0}_{\delta}}.
\end{align*}
We wish to establish that
\begin{equation}\label{eq:mod_mult_est_time}
    \Vert I(\vert u\vert^{2k}u)\Vert_{X^{2,0}_{\delta}}\lesssim\delta^{\frac{1}{2}-}\Vert Iu\Vert_{X^{2,\frac{1}{2}+}_{\delta}}^{2k+1}.
\end{equation}
Before doing so, we exploit \eqref{eq:mod_mult_est_time} so as to obtain
\begin{equation}\label{eq:exp_gain_delta}
     \Vert\Gamma(Iu)\Vert_{X^{2,\frac{1}{2}+}}\lesssim\Vert Iu_{0}\Vert_{H^{2}_{x}}+\delta^{1-}\Vert Iu\Vert_{X^{2,\frac{1}{2}+}_{\delta}}^{2k+1}.
\end{equation}
To see \eqref{eq:mod_mult_est_time}, we shall follow the argument from Proposition 4.10 in \cite{erdoğan_tzirakis_2016}. By duality and ignoring the $\delta$-dependence, we have
\begin{equation*}
    \Vert I(\vert u\vert^{2k}u)\Vert_{X^{2,0}}=\sup_{\Vert f\Vert_{X^{-2,0}}=1}\Big\vert\langle I(\vert u\vert^{2k}u),f\rangle_{L^{2}(\mathbb{R}\times\mathbb{R}^{2})}\Big\vert.
\end{equation*}
Using Plancherel's theorem, we have
\begin{multline*}
    \int\int I(\vert u\vert^{2k}u)(t,x)\overline{f(t,x)}\text{d}x\text{d}t=\\\int_{\sum_{i=1}^{2k+2}\tau_{i}=0}\int_{\sum_{i=1}^{2k+2}\xi_{i}=0}m(\xi_{2k+2})\widehat{u}(\tau_{1},\xi_{1})\reallywidehat{\overline{u}}(\tau_{2},\xi_{2})\dots\widehat{u}(\tau_{2k+1},\xi_{2k+1})\reallywidehat{\overline{f}}(\tau_{2k+2},\xi_{2k+2}).
\end{multline*}
Define
\begin{align*}
    f_{1}(\tau_{1},\xi_{1})=&\vert\widehat{u}(\tau_{1},\xi_{1})\vert m(\xi_{1})\langle\xi_{1}\rangle^{2}\langle\tau_{1}-\vert\xi_{1}\vert^{4}-\vert\xi_{1}\vert^{2}\rangle^{\frac{1}{2}+}\\
    f_{2}(\tau_{2},\xi_{2})=&\vert\widehat{u}(-\tau_{2},-\xi_{2})\vert m(\xi_{2})\langle\xi_{2}\rangle^{2}\langle\tau_{2}+\vert\xi_{2}\vert^{4}+\vert\xi_{2}\vert^{2}\rangle^{\frac{1}{2}+}\\
    \vdots&\\
    f_{2k+1}(\tau_{2k+1},\xi_{2k+1})=&\vert\widehat{u}(\tau_{2k+1},\xi_{2k+1})\vert m(\xi_{2k+1})\langle\xi_{2k+1}\rangle^{2}\langle\tau_{2k+1}-\vert\xi_{2k+1}\vert^{4}-\vert\xi_{2k+1}\vert^{2}\rangle^{\frac{1}{2}+}\\
    f_{2k+2}(\tau_{2k+2},\xi_{2k+2})=&\vert\widehat{f}(\tau_{2k+2},\xi_{2k+2})\vert\langle\xi_{2k+2}\rangle^{-2}.
\end{align*}
Therefore, \eqref{eq:mod_mult_est_time} amounts to showing that 
\begin{equation*}
\int_{\sum_{i=1}^{2k+2}\tau_{i}=0}\int_{\sum_{i=1}^{2k+2}\xi_{i}=0}\frac{m(\xi_{2k+2})\langle\xi_{2k+2}\rangle^{2}}{\prod_{i=1}^{2k+1}m(\xi_{i})\langle\xi_{i}\rangle^{2}\langle\tau_{i}-(-1)^{i-1}(\vert\xi_{i}\vert^{4}+\vert\xi_{i}\vert^{2})\rangle^{\frac{1}{2}+}}\prod_{i=1}^{2k+2}f_{i}(\tau_{i},\xi_{i})\lesssim\prod_{i=1}^{2k+2}\Vert f_{i}\Vert_{L^{2}_{\tau,\xi}}.
\end{equation*}
The function $m(x)x^{2-s}$ is increasing so that on the hyperplane $\sum_{i=1}^{2k+1}\xi_{i}=0$, we have
\begin{equation}\label{eq:hyperplane_est}
\frac{m(\xi_{2k+2})\langle\xi_{2k+2}\rangle^{2-s}}{\prod_{i=1}^{2k+1}m(\xi_{i})\langle\xi_{i}\rangle^{2-s}}\lesssim 1.
\end{equation}
Thus, using \eqref{eq:hyperplane_est}, it suffices to show that
\begin{equation*}
\int_{\sum_{i=1}^{2k+2}\tau_{i}=0}\int_{\sum_{i=1}^{2k+2}\xi_{i}=0}\frac{\langle\xi_{2k+2}\rangle^{s}}{\prod_{i=1}^{2k+1}\langle\xi_{i}\rangle^{s}\langle\tau_{i}-(-1)^{i-1}(\vert\xi_{i}\vert^{4}+\vert\xi_{i}\vert^{2})\rangle^{\frac{1}{2}+}}\prod_{i=1}^{2k+2}f_{i}(\tau_{i},\xi_{i})\lesssim\prod_{i=1}^{2k+2}\Vert f_{i}\Vert_{L^{2}_{\tau,\xi}}.
\end{equation*}
Recall for $s\geq 0$ and $\sum_{i=1}^{2k+2}\xi_{i}=0$ that
\begin{equation*}
    \langle\xi_{2k+2}\rangle^{s}\lesssim\sum_{i=1}^{2k+1}\langle\xi_{i}\rangle^{s}
\end{equation*}
and by the symmetry in the frequency variables $\xi_{1},\dots,\xi_{2k+1}$, one may assume that $\vert\xi_{1}\vert\geq\dots\geq\vert\xi_{2k+1}\vert$. Hence, it is enough to prove that
\begin{equation}\label{eq:mult_disp_est}
\int_{\sum_{i=1}^{2k+2}\tau_{i}=0}\int_{\sum_{i=1}^{2k+2}\xi_{i}=0}\frac{\langle\xi_{1}\rangle^{s}}{\prod_{i=1}^{2k+1}\langle\xi_{i}\rangle^{s}\langle\tau_{i}-(-1)^{i-1}(\vert\xi_{i}\vert^{4}+\vert\xi_{i}\vert^{2})\rangle^{\frac{1}{2}+}}\prod_{i=1}^{2k+2}f_{i}(\tau_{i},\xi_{i})\lesssim\prod_{i=1}^{2k+2}\Vert f_{i}\Vert_{L^{2}_{\tau,\xi}}.
\end{equation}
Reverting back the Plancherel formula to come back to the physical space-time variables $x,t$, and using duality, the inequality in \eqref{eq:mult_disp_est} is equivalent to showing that
\begin{equation*}
    \Vert(\langle\nabla\rangle^{s}h_{1})h_{2}\dots h_{2k+1}\Vert_{L^{2}_{t,x}}\lesssim\prod_{j=1}^{2k+1}\Vert h_{j}\Vert_{X^{s,\frac{1}{2}+}}
\end{equation*}
where
\begin{equation*}
    h_{i}(t,x)=\mathcal{F}^{-1}_{\tau,\xi}\big(\langle\xi_{i}\rangle^{-s}\langle\tau_{i}-(-1)^{i-1}(\vert\xi_{i}\vert^{4}+\vert\xi_{i}\vert^{2})\rangle^{-\frac{1}{2}-}f_{i}(\tau_{i},\xi_{i})\big)(t,x)
\end{equation*}
for $1\leq i\leq 2k+1$. Now, using the following continuous embeddings
\begin{equation*}
    X^{1+,\frac{1}{2}+}(\mathbb{R}\times\mathbb{R}^{2})\hookrightarrow C^{0}_{t}H^{1+}_{x}(\mathbb{R}\times\mathbb{R}^{2})\hookrightarrow C^{0}_{t}C^{0}_{x}(\mathbb{R}\times\mathbb{R}^{2})
\end{equation*}
we get
\begin{equation}\label{eq:embedding}
    \Vert h\Vert_{L^{\infty}_{t,x}}\lesssim \Vert h\Vert_{X^{1+,\frac{1}{2}+}}.
\end{equation}
Therefore, using Hölder inequality and \eqref{eq:embedding}, we arrive at
\begin{align}
    \Vert(\langle\nabla\rangle^{s}h_{1})h_{2}\dots h_{2k+1}\Vert_{L^{2}_{t,x}}\leq\Vert\langle\nabla\rangle^{s}h_{1}\Vert_{L^{2}_{t,x}}\prod_{i=2}^{2k+1}\Vert h_{i}\Vert_{L^{\infty}_{t,x}}&\lesssim\Vert h_{1}\Vert_{X^{s,0}}\prod_{i=2}^{2k+1}\Vert h_{i}\Vert_{X^{1+,\frac{1}{2}+}}\nonumber\\
    &\lesssim\Vert h_{1}\Vert_{X^{s,0}}\prod_{i=2}^{2k+1}\Vert h_{i}\Vert_{X^{s,\frac{1}{2}+}}.\label{eq:cts_emb_mult_est}
\end{align}
Now, the $\delta$-dependence can be taken into account in \eqref{eq:cts_emb_mult_est}. In view of Lemma 3.11 in \cite{erdoğan_tzirakis_2016}, we have
\begin{equation*}
    \Vert h_{1}\Vert_{X^{s,0}_{\delta}}\lesssim\delta^{\frac{1}{2}-}\Vert h_{1}\Vert_{X^{s,\frac{1}{2}-}_{\delta}}\lesssim\delta^{\frac{1}{2}-}\Vert h_{1}\Vert_{X^{s,\frac{1}{2}+}_{\delta}}.
\end{equation*}
So, since the inequality \eqref{eq:cts_emb_mult_est} holds for restricted norms as well, we have
\begin{equation*}
    \Vert(\langle\nabla\rangle^{s}h_{1})h_{2}\dots h_{2k+1}\Vert_{L^{2}_{t\in[0,\delta],x}}\lesssim\delta^{\frac{1}{2}-}\prod_{i=1}^{2k+1}\Vert h_{i}\Vert_{X^{s,\frac{1}{2}+}_{\delta}}=\delta^{\frac{1}{2}-}\prod_{i=1}^{2k+1}\Vert f_{i}\Vert_{X^{0,0}_{\delta}}=\delta^{\frac{1}{2}-}\prod_{i=1}^{2k+1}\Vert Iu\Vert_{X^{2,\frac{1}{2}+}_{\delta}}.
\end{equation*}
This completes the proof of \eqref{eq:mod_mult_est_time}. 

\noindent Proceeding as in the estimate of \eqref{eq:exp_gain_delta}, we may obtain the analogous estimate for the difference
\begin{align*}
    \Vert\Gamma(Iu)-\Gamma(Iv)\Vert_{X^{2,\frac{1}{2}+}}&\leq c'\delta^{1-}(\Vert Iu\Vert_{X^{2,\frac{1}{2}+}}^{2k}+\Vert Iv\Vert_{X^{2,\frac{1}{2}+}}^{2k})\Vert Iu-Iv\Vert_{X^{2,\frac{1}{2}+}}.
\end{align*}
Define $B(r)=\{Iu\in X^{2,\frac{1}{2}+}:\Vert Iu\Vert_{X^{2,\frac{1}{2}+}}\leq r\}$ where $r=\Vert Iu_{0}\Vert_{H^{2}_{x}}$. Choosing
\begin{equation*}
    0<\delta^{1-}<\min\{1,\frac{1}{2c'r^{2k}}\},
\end{equation*}
we guarantee that $\Gamma:B(r)\to B(r)$ is a contraction map. Thus, there exists a unique solution $Iu\in B(r)$ to the equation \eqref{eq:mod_4NLS} such that $\Gamma(Iu)=Iu$. Moreover, the restriction $Iu_{\delta}=Iu|_{[0,\delta]}\in C([0,\delta],H^{2}(\mathbb{R}^{2}))\cap X^{2,\frac{1}{2}+}_{\delta}$ is a solution of the modified equation \eqref{eq:mod_4NLS} in $[0,\delta]$ and the proof is complete.
\end{proof}
\begin{proposition}[Almost Conservation Law]\label{almost_cons_law}
Suppose $s>2-\frac{3}{4k}$, $N\gg 1$, and $Iu$ is a solution of \eqref{eq:mod_4NLS} on $[0,\delta]$ established in Lemma \ref{mod_local_exst}. Then the following estimate holds:
\begin{equation}
    E(Iu(\delta))-E(Iu(0))\lesssim N^{-3+}\Vert Iu\Vert_{X^{2,\frac{1}{2}+}_{\delta}}^{2k+2}+N^{-4+}\Vert Iu\Vert_{X^{2,\frac{1}{2}+}_{\delta}}^{4k+2}.
\end{equation}
\end{proposition}
\begin{proof} We start by estimating $Term_{1}$. Recall that
\begin{equation*}
    Term_{1}=\Big|\int_{0}^{\delta}\int_{\sum_{i=1}^{2k+2}\xi_{i}=0}\Big(1-\frac{m(\xi_{2\dots 2k+2})}{m(\xi_{2})\dots m(\xi_{2k+2})}\Big)\reallywidehat{\overline{\Delta^{2}Iu}}(\xi_{1})\widehat{Iu}(\xi_{2})\widehat{\overline{Iu}}(\xi_{3})\dots\widehat{Iu}(\xi_{2k+2})\Big|.
\end{equation*}
Let us define
\begin{align*}
    \reallywidehat{u_{N_{1}}}(\xi_{1})&=\reallywidehat{\overline{P_{N_{1}}\Delta^{2} Iu}}(\xi_{1}),\\
    \reallywidehat{u_{N_{j}}}(\xi_{j})&=\reallywidehat{P_{N_{j}}Iu}(\xi_{j}),\quad j=2,4,6,\dots,2k+2,\\
    \reallywidehat{u_{N_{l}}}(\xi_{l})&=\reallywidehat{\overline{P_{N_{l}}Iu}}(\xi_{l})\quad l=3,5,7,\dots,2k+1
\end{align*}
where $N_{i}=2^{k_{i}}$ for $k_{i}\in\mathbb{N}$, $i=1,\dots,2k+2$. Then
\begin{equation*}
    Term_{1}\leq\sum_{N_{1},\dots,N_{2k+2}}\Big|\int_{0}^{\delta}\int_{\sum_{i=1}^{2k+2}\xi_{i}=0}\big(1-\frac{m_{2\dots 2k+2}}{m_{2}\dots m_{2k+2}}\big)\widehat{u_{N_{1}}}(\xi_{1})\dots\widehat{u_{N_{2k+2}}}(\xi_{2k+2})\Big|.
\end{equation*}
By the symmetry in $\xi_{2},\dots,\xi_{2k+2}$ variables, we may assume that $N_{2}\geq\dots\geq N_{2k+2}$ and this implies $N_{1}\lesssim N_{2}$. Also, we may assume that the spatial Fourier transform of dyadic pieces are nonnegative. By this assumption, we can take the multiplier out of the integral with a pointwise bound in the different frequency interaction cases. Thus, it is enough to show that
\begin{equation}\label{eq:term1_dyad_pieces}
    \Big|\int_{0}^{\delta}\int_{\sum_{i=1}^{2k+2}\xi_{i}=0}\big(1-\frac{m_{2\dots 2k+2}}{m_{2}\dots m_{2k+2}}\big)\widehat{u_{N_{1}}}(\xi_{1})\dots\widehat{u_{N_{2k+2}}}(\xi_{2k+2})\Big|\lesssim N^{-3+}N_{2}^{0-}\Vert u_{N_{1}}\Vert_{X^{-2,\frac{1}{2}+}}\prod_{j=2}^{2k+2}\Vert u_{N_{j}}\Vert_{X^{2,\frac{1}{2}+}}.
\end{equation}
There are three frequency interaction cases to consider:\\
\noindent
\textbf{Case 1: }$N_{1}\sim N_{2}\gtrsim N\gg N_{3}$.\\
\noindent
In this case, the bound on the multiplier is given by
\begin{equation*}
    \Big|1-\frac{m_{2\dots 2k+2}}{m_{2}\dots m_{2k+2}}\Big|\lesssim\frac{N_{3}}{N_{2}}.
\end{equation*}
By Hölder's inequality, we have
\begin{equation*}
    \text{LHS of \eqref{eq:term1_dyad_pieces}}\lesssim\frac{N_{3}}{N_{2}}\Vert u_{N_{1}}\Vert_{L^{2}_{t}L^{\infty}_{x}}\Vert u_{N_{2}}\Vert_{L^{2}_{t}L^{\infty}_{x}}\prod_{j=3}^{2k+2}\Vert u_{N_{j}}\Vert_{L^{\infty}_{t}L^{2k}_{x}}.
\end{equation*}
For the first two factors, we use \eqref{keyestimate} with $\mu=1$ to get
\begin{equation*}
    \Vert u_{N_{i}}\Vert_{L^{2}_{t}L^{\infty}_{x}}\lesssim N_{i}^{-1}\Vert u_{N_{i}}\Vert_{X^{0,\frac{1}{2}+}_{\delta}},\quad i=1,2.
\end{equation*}
The remaining factors can be dealt with by using the Sobolev embedding and the Strichartz estimate for biharmonic admissible pairs:
\begin{equation*}
    \Vert u_{N_{j}}\Vert_{L^{\infty}_{t}L^{2k}_{x}}\lesssim\Vert\langle\nabla\rangle^{\frac{k-1}{k}}u_{N_{j}}\Vert_{L^{\infty}_{t}L^{2}_{x}}\lesssim N_{j}^{\frac{k-1}{k}}\Vert u_{N_{j}}\Vert_{X^{0,\frac{1}{2}+}_{\delta}},\quad j=3,\dots,2k+2.
\end{equation*}
Using the above inequalities, we obtain
\begin{equation*}
    \text{LHS of \eqref{eq:term1_dyad_pieces}}\lesssim\frac{N_{3}}{N_{2}}\frac{\prod_{j=3}^{2k+2}N_{j}^{\frac{k-1}{k}}}{N_{1}N_{2}}\frac{N_{1}^{2}}{\prod_{j=2}^{2k+2}N_{j}^{2}}\Vert u_{N_{1}}\Vert_{X^{-2,\frac{1}{2}+}_{\delta}}\prod_{j=2}^{2k+2}\Vert u_{N_{j}}\Vert_{X^{2,\frac{1}{2}+}_{\delta}}.
\end{equation*}
Thus, it is enough to demonstrate that
\begin{equation*}
    \frac{1}{N_{2}^{3}N_{3}^{1/k}\prod_{j=4}^{2k+2}N_{j}^{(k+1)/k}}\lesssim N^{-3+}N_{2}^{0-}.
\end{equation*}
Since $N_{j}\geq 1$ for all $j$ and $N_{2}\gtrsim N$, by a crude estimate, the above inequality follows by  
\begin{equation*}
    N^{3-}N_{2}^{-3+}\lesssim 1.
\end{equation*}
\noindent
\textbf{Case 2: }$N_{2}\sim N_{3}\gtrsim N$.\\
\noindent
In this case, the bound for the multiplier is given by
\begin{equation*}
    \Big|1-\frac{m_{2\dots 2k+2}}{m_{2}\dots m_{2k+2}}\Big|\lesssim\frac{m(N_{1})}{m(N_{2})\dots m(N_{2k+2})}.
\end{equation*}
Using Hölder's inequality, we have
\begin{equation*}
    \text{LHS of \eqref{eq:term1_dyad_pieces}}\lesssim\frac{m(N_{1})}{m(N_{2})\dots m(N_{2k+2})}\Vert u_{N_{1}}\Vert_{L^{\infty}_{t}L^{2}_{x}}\Vert u_{N_{2}}\Vert_{L^{2}_{t}L^{\infty}_{x}}\Vert u_{N_{3}}\Vert_{L^{2}_{t}L^{\infty}_{x}}\prod_{j=4}^{2k+2}\Vert u_{N_{j}}\Vert_{L^{\infty}_{t}L^{4k-2}_{x}}.
\end{equation*}
We implement \eqref{keyestimate} with $\mu=1$ for $u_{N_{2}},u_{N_{3}}$
\begin{equation*}
    \Vert u_{N_{i}}\Vert_{L^{2}_{t}L^{\infty}_{x}}\lesssim N_{i}^{-1}\Vert u_{N_{i}}\Vert_{X^{0,\frac{1}{2}+}_{\delta}},\quad i=2,3.
\end{equation*}
For the first factor, we apply \eqref{keyestimate} with $\mu=0$
\begin{equation*}
    \Vert u_{N_{1}}\Vert_{L^{\infty}_{t}L^{2}_{x}}\lesssim\Vert u_{N_{1}}\Vert_{X^{0,\frac{1}{2}+}_{\delta}}.
\end{equation*}
The remaining factors can be handled via Sobolev embedding and Strichartz estimate for biharmonic admissible pairs (i.e., \eqref{keyestimate} with $\mu=0$)
\begin{equation*}
    \Vert u_{N_{j}}\Vert_{L^{\infty}_{t}L^{4k-2}_{x}}\lesssim\Vert\langle\nabla\rangle^{\frac{2k-2}{2k-1}}u_{N_{j}}\Vert_{L^{\infty}_{t}L^{2}_{x}}\lesssim\Vert u_{N_{j}}\Vert_{X^{\frac{2k-2}{2k-1},\frac{1}{2}+}_{\delta}},\quad j=4,\dots,2k+2.
\end{equation*}
Combining all of the above estimates together with \eqref{eq:Berns_ineq}, we obtain
\begin{align*}
    \text{LHS of \eqref{eq:term1_dyad_pieces}}&\lesssim\frac{m(N_{1})}{m(N_{2})\dots m(N_{2k+2})}\frac{\langle N_{1}\rangle^{2}}{N_{2}N_{3}\prod_{j=2}^{2k+2}\langle N_{j}\rangle^{2}}\prod_{j=4}^{2k+2}\langle N_{j}\rangle^{\frac{2k-2}{2k-1}}\Vert u_{N_{1}}\Vert_{X^{-2,\frac{1}{2}+}_{\delta}}\prod_{j=2}^{2k+2}\Vert u_{N_{j}}\Vert_{X^{2,\frac{1}{2}+}_{\delta}}\\
    &\lesssim\frac{m(N_{1})}{m(N_{2})\dots m(N_{2k+2})}\frac{N_{1}^{2}}{N_{2}^{3}N_{3}^{3}\prod_{j=4}^{2k+2}N_{j}^{\frac{2k}{2k-1}}}\Vert u_{N_{1}}\Vert_{X^{-2,\frac{1}{2}+}_{\delta}}\prod_{j=2}^{2k+2}\Vert u_{N_{j}}\Vert_{X^{2,\frac{1}{2}+}_{\delta}}.
\end{align*}
As $m(x)x^{p}$ is increasing for $x\geq 0$ and $p\geq\frac{3}{4k}$ it suffices to show that
\begin{equation*}
    \frac{m(N_{1})N_{1}^{2}N^{3-}N_{2}^{0+}}{(m(N_{2}))^{2}N_{2}^{6}}\lesssim 1.
\end{equation*}
This can be seen since $m(N_{1})N_{1}^{2}\lesssim m(N_{2})N_{2}^{2}$ and $m(N_{2})N_{2}^{4-}\gtrsim N^{4-}$ implies that
\begin{equation*}
    \frac{m(N_{1})N_{1}^{2}N^{3-}N_{2}^{0+}}{(m(N_{2}))^{2}N_{2}^{6}}\lesssim N^{3-}N^{-4+}\lesssim 1.
\end{equation*}\\
\noindent
\textbf{Case 3: }$N_{2}\gg N_{3}\gtrsim N$.\\
\noindent
In this case, $N_{1}\sim N_{2}$. Multiplier can be estimated by
\begin{equation*}
    \Big|1-\frac{m_{2\dots 2k+2}}{m_{2}\dots m_{2k+2}}\Big|\lesssim\frac{m(N_{1})}{m(N_{2})\dots m(N_{2k+2})}.
\end{equation*}
Thanks to the non-negativity assumption of the Fourier transform of each piece, we are able now to turn back to the physical side in which we employ a series of Hölder, Strichartz \eqref{keyestimate}, and the Sobolev estimates as in Case 1 to get the following bound for the left side of \eqref{eq:term1_dyad_pieces}:\begin{align*}
    &\frac{m(N_1)}{m(N_{2})\dots m(N_{2k+2})}\Vert u_{N_{1}}\Vert_{L^{2}_{t}L^{\infty}_{x}}\Vert u_{N_{2}}\Vert_{L^{2}_{t}L^{\infty}_{x}}\prod_{j=3}^{2k+2}\Vert u_{N_{j}}\Vert_{L^{\infty}_{t}L^{2k}_{x}}\\&\lesssim \frac{m(N_1)}{m(N_{2})\dots m(N_{2k+2})} \frac{N_1^2}{N_1N_2^3}\frac{\prod_{j=3}^{2k+2}\langle N_j\rangle^{\frac{k-1}{k}}}{\prod_{j=3}^{2k+2}\langle N_j\rangle^2}\Vert u_{N_{1}}\Vert_{X^{-2,\frac{1}{2}+}_{\delta}}\prod_{j=2}^{2k+2}\Vert u_{N_{j}}\Vert_{X^{2,\frac{1}{2}+}_{\delta}}\\&\lesssim \frac{m(N_1)N_1^2\times(N_1N_2)^{-1}}{m(N_{2})N_2^2\,\,m(N_{3})N_3^{\frac{k+1}{k}}m(N_{4})N_{4}^{\frac{k+1}{k}}\dots m(N_{2k+2}) N_{2k+2}^{\frac{k+1}{k}}} \Vert u_{N_{1}}\Vert_{X^{-2,\frac{1}{2}+}_{\delta}}\prod_{j=2}^{2k+2}\Vert u_{N_{j}}\Vert_{X^{2,\frac{1}{2}+}_{\delta}} \\&\lesssim \frac{(N_1N_2)^{-1}}{N^{\frac{k+1}{k}}}\Vert u_{N_{1}}\Vert_{X^{-2,\frac{1}{2}+}_{\delta}}\prod_{j=2}^{2k+2}\Vert u_{N_{j}}\Vert_{X^{2,\frac{1}{2}+}_{\delta}}
\end{align*} 
where in the last line we use the fact that $N_1\sim N_2\Rightarrow m(N_1)N_1^2\sim m(N_2)N_2^2$ and that $m(x)x^p$ is increasing and bounded below for $s+p\geq 2$. Hence, it suffices to show that
\begin{equation*}
    \frac{N^{3-}N_{2}^{-2+}}{N^{\frac{k+1}{k}}}=N^{2-\frac{1}{k}-}N_{2}^{-2+}\lesssim 1
\end{equation*}
which indeed is true for all positive integers $k$. This completes estimating $Term_{1}$. 

\noindent For the second term, recall that
\begin{equation*}
    Term_{2}=\Big|\int_{0}^{\delta}\int_{\sum_{i=1}^{2k+2}\xi_{i}=0}\Big(1-\frac{m_{2\dots 2k+2}}{m_{2}\dots m_{2k+2}})\reallywidehat{\overline{\Delta Iu}}(\xi_{1})\dots\widehat{Iu}(\xi_{2k+2})\Big|.
\end{equation*}
Define
\begin{align*}
    \reallywidehat{u_{N_{1}}}(\xi_{1})&=\reallywidehat{\overline{P_{N_{1}}\Delta Iu}}(\xi_{1}),\\
    \reallywidehat{u_{N_{j}}}(\xi_{j})&=\reallywidehat{P_{N_{j}}Iu}(\xi_{j}),\quad j=2,4,6,\dots,2k+2,\\
    \reallywidehat{u_{N_{l}}}(\xi_{l})&=\reallywidehat{\overline{P_{N_{l}}Iu}}(\xi_{l}),\quad l=3,5,7,\dots,2k+1
\end{align*}
where $N_{i}=2^{k_{i}}$ for $k_{i}\in\mathbb{N}$, $i=1,\dots,2k+2$. Assuming the nonnegativity of the spatial Fourier transform of dyadic pieces, we can take the multiplier out of the integral with a pointwise bound in different frequency interaction cases. Also by symmetry, we may assume that $N_{2}\geq\dots\geq N_{2k+2}$ which implies $N_{1}\lesssim N_{2}$. In the end, we wish to show that
\begin{equation}\label{eq:term_2_dyad_pieces}
    \Big|\int_{0}^{\delta}\int_{\sum_{i=1}^{2k+2}}\reallywidehat{u_{N_1}}(\xi_{1})\dots\widehat{u_{N_{2k+2}}}(\xi_{2k+2})\Big|\lesssim N_{2}^{0-}N^{-3+}\Vert u_{N_{1}}\Vert_{X^{0,\frac{1}{2}+}_{\delta}}\prod_{j=2}^{2k+2}\Vert u_{N_{j}}\Vert_{X^{2,\frac{1}{2}+}_{\delta}}.
\end{equation}
To achieve the desired result, we need to cope with the following three frequency interaction cases:\\
\noindent
\textbf{Case 1: }$N_{1}\sim N_{2}\gtrsim N\gg N_{3}$.\\
\noindent
In this case, the bound on the multiplier is given by
\begin{equation*}
    \Big|1-\frac{m_{2\dots 2k+2}}{m_{2}\dots m_{2k+2}}\Big|\lesssim\frac{N_{3}}{N_{2}}.
\end{equation*}
Utilizing the above bound together with Hölder's inequality after undoing the Plancherel formula, we have
\begin{equation*}
    \text{LHS of \eqref{eq:term_2_dyad_pieces}}\lesssim\frac{N_{3}}{N_{2}}\Vert u_{N_{1}}\Vert_{L^{2}_{t}L^{\infty}_{x}}\Vert u_{N_{2}}\Vert_{L^{2}_{t}L^{\infty}_{x}}\prod_{j=3}^{2k+2}\Vert u_{N_{j}}\Vert_{L^{\infty}_{t}L^{2k}_{x}}.
\end{equation*}
Using the inequality \eqref{keyestimate} with $\mu=1$ for the first two factors, we get
\begin{equation*}
    \Vert u_{N_{i}}\Vert_{L^{2}_{t}L^{\infty}_{x}}\lesssim N_{i}^{-1}\Vert u_{N_{i}}\Vert_{X^{0,\frac{1}{2}+}_{\delta}},\quad i=1,2.
\end{equation*}
For the remaining factors, we first apply Sobolev embedding and then \eqref{keyestimate} with $\mu=0$ to obtain
\begin{equation*}
    \Vert u_{N_{j}}\Vert_{L^{\infty}_{t}L^{2k}_{x}}\lesssim\Vert\langle\nabla\rangle^{\frac{k-1}{k}} u_{N_{j}}\Vert_{L^{\infty}_{t}L^{2}_{x}}\lesssim N^{\frac{k-1}{k}}_{j}\Vert u_{N_{j}}\Vert_{X^{0,\frac{1}{2}+}_{\delta}},\quad j=3,\dots,2k+2.
\end{equation*}
Combining the above two bounds, we get
\begin{equation*}
    \text{LHS of \eqref{eq:term_2_dyad_pieces}}\lesssim\frac{N_{3}}{N_{1}N_{2}^{2}}\frac{1}{N_{2}^{2}N_{3}^{\frac{k+1}{k}}\dots N_{2k+2}^{\frac{k+1}{k}}}\Vert u_{N_{1}}\Vert_{X^{0,\frac{1}{2}+}_{\delta}}\prod_{j=2}^{2k+2}\Vert u_{N_{j}}\Vert_{X^{2,\frac{1}{2}+}_{\delta}}.
\end{equation*}
Then it is enough to show that
\begin{equation*}
    N^{3-}N_{2}^{-5+}\lesssim 1
\end{equation*}
which is true as $N_{2}\gtrsim N$.\\
\noindent
\textbf{Case 2: }$N_{2}\sim N_{3}\gtrsim N$.\\
\noindent
In this case, the bound for the multiplier is calculated by
\begin{equation*}
    \Big|1-\frac{m_{2\dots 2k+2}}{m_{2}\dots m_{2k+2}}\Big|\lesssim\frac{m(N_{1})}{m(N_{2})m(N_{3})\dots m(N_{2k+2})}.
\end{equation*}
Applying this bound with Hölder's inequality by means of reverting Plancherel's formula back to physical space in the spatial variables, we get
\begin{equation*}
    \text{LHS of \eqref{eq:term_2_dyad_pieces}}\lesssim\frac{m(N_{1})}{m(N_{2})\dots m(N_{2k+2})}\Vert u_{N_{1}}\Vert_{L^{\infty}_{t}L^{2k}_{x}}\Vert u_{N_{2}}\Vert_{L^{2}_{t}L^{\infty}_{x}}\Vert u_{N_{3}}\Vert_{L^{2}_{t}L^{\infty}_{x}}\prod_{j=4}^{2k+2}\Vert u_{N_{j}}\Vert_{L^{\infty}_{t}L^{2k}_{x}}.
\end{equation*}
For the first factor and the factors in the product, we apply Sobolev embedding and \eqref{keyestimate} with $\mu=0$ to obtain
\begin{equation*}
    \Vert u_{N_{i}}\Vert_{L^{\infty}_{t}L^{2k}_{x}}\lesssim\Vert\langle\nabla\rangle^{\frac{k-1}{k}} u_{N_{i}}\Vert_{L^{\infty}_{t}L^{2}_{x}}\lesssim N_{i}^{\frac{k-1}{k}}\Vert u_{N_{i}}\Vert_{X^{0,\frac{1}{2}+}_{\delta}},\quad i=1,4,5,\dots,2k+2.
\end{equation*}
For the second and the third factors, we implement \eqref{keyestimate} with $\mu=1$ to have
\begin{equation*}
    \Vert u_{N_{i}}\Vert_{L^{2}_{t}L^{\infty}_{x}}\lesssim N_{i}^{-1}\Vert u_{N_{i}}\Vert_{X^{0,\frac{1}{2}+}_{\delta}}.
\end{equation*}
Therefore, we have
\begin{equation*}
    \text{LHS of \eqref{eq:term_2_dyad_pieces}}\lesssim\frac{m(N_{1})}{m(N_{2})\dots m(N_{2k+2})}\frac{N_{1}^{\frac{k-1}{k}}}{N_{2}^{3}N_{3}^{3}\prod_{j=4}^{2k+2}N_{j}^{\frac{k+1}{k}}}\Vert u_{N_{1}}\Vert_{X^{0,\frac{1}{2}+}_{\delta}}\prod_{j=2}^{2k+2}\Vert u_{N_{i}}\Vert_{X^{2,\frac{1}{2}+}_{\delta}}.
\end{equation*}
Since $m(x)x^{\frac{k+1}{k}}$ is increasing for $0\leq x$ and $0\leq m\leq 1$, it is enough to verify that
\begin{equation*}
    \frac{m(N_{1})N_{1}^{\frac{k-1}{k}}N_{2}^{0+}N^{3-}}{m(N_{2})^{2}N_{2}^{6}}\lesssim 1.
\end{equation*}
 Then using $m(N_{1})N_{1}^{\frac{k-1}{k}}\lesssim m(N_{2})N_{2}^{\frac{k-1}{k}}$ and $m(N_{2})N_{2}^{\frac{k+1}{k}}\gtrsim N^{\frac{k+1}{k}}$, the above inequality boils down to
\begin{equation*}
    N^{1-\frac{2}{k}-}N_{2}^{-3+\frac{1}{k}+}\lesssim 1
\end{equation*}
which holds by our restrictions on the frequencies.
\\ \noindent
\textbf{Case 3: }$N_{1}\sim N_{2}\gg N_{3}\gtrsim N$.\\
\noindent
In this case, the bound on the multiplier is
\begin{equation*}
    \Big|1-\frac{m_{2\dots2k+2}}{m_{2}\dots m_{2k+2}}\Big|\lesssim\frac{m(N_{1})}{m(N_{2})m(N_{3})\dots m(N_{2k+2})}\lesssim\frac{1}{m(N_{3})\dots m(N_{2k+2})}.
\end{equation*}
Applying the bound for the multiplier, reverting back Plancherel formula, and exploiting Hölder's inequality, we get
\begin{equation*}
    \text{LHS of \eqref{eq:term_2_dyad_pieces}}\lesssim \frac{1}{m(N_{3})\dots m(N_{2k+2})}\Vert u_{N_{1}}\Vert_{L^{2}_{t}L^{\infty}_{x}}\Vert u_{N_{2}}\Vert_{L^{2}_{t}L^{\infty}_{x}}\prod_{j=3}^{2k+2}\Vert u_{N_{j}}\Vert_{L^{\infty}_{t}L^{2k}_{x}}.
\end{equation*}
Continuing similarly as in Case 2, we obtain
\begin{equation*}
    \text{LHS of \eqref{eq:term_2_dyad_pieces}}\lesssim\frac{1}{m(N_{3})\dots m(N_{2k+2})}\frac{1}{N_{1}N_{2}}\Vert u_{N_{1}}\Vert_{X^{0,\frac{1}{2}+}_{\delta}}\Vert u_{N_{2}}\Vert_{X^{0,\frac{1}{2}+}_{\delta}}\prod_{j=3}^{2k+2}N_{j}^{\frac{k-1}{k}}\Vert u_{N_{j}}\Vert_{X^{0,\frac{1}{2}+}_{\delta}}.
\end{equation*}
Then it suffices to show
\begin{equation*}
    \frac{N^{3-}N_{2}^{0+}}{N_{2}^{4}\prod_{j=3}^{2k+2}N^{\frac{k+1}{k}}_{j}m(N_{j})}\lesssim 1.
\end{equation*}
As $m(x)x^{\frac{k+1}{k}}$ is an increasing function, we replace $m(N_{3})N_{3}^{\frac{k+1}{k}}$ with $m(N)N^{\frac{k+1}{k}}$, and $m(N_{i})N_{i}^{\frac{k+1}{k}}$ with $1$ for $i=4,\dots,2k+2$. Then the above bound reduces to 
\begin{equation*}
    N^{3-\frac{k+1}{k}-}N_{2}^{-4+}\lesssim 1
\end{equation*}
which holds as $N_{2}\gtrsim N$. \\ It remains to estimate $Term_{3}$. Recall that
\begin{equation*}
    Term_{3}=\Big|\int_{0}^{\delta}\int_{\sum_{i=1}^{2k+2}\xi=0}\Big(1-\frac{m_{2\dots2k+2}}{m_{2}\dots m_{2k+2}}\Big)\reallywidehat{\overline{I(\vert u\vert^{2k}u)}}(\xi_{1})\dots\widehat{Iu}(\xi_{2k+2})\Big|.
\end{equation*}
Revealing the convolution structure of $\reallywidehat{\overline{I(\vert u\vert^{2k}u)}}$, relabelling the frequency variables, and defining
\begin{align*}
    \reallywidehat{u_{N_{j}}}(\xi_{j})&=\reallywidehat{P_{N_{j}}u}(\xi_{j}),\quad j=2k+2,2k+4,\dots,4k+2,\\
    \reallywidehat{u_{N_{l}}}(\xi_{l})&=\reallywidehat{\overline{P_{N_{l}}u}}(\xi_{l}),\quad l=2k+3,2k+5,\dots,4k+1,
\end{align*}
we can bound $Term_{3}$ by
\begin{multline*}
    Term_{3}\leq\sum_{N_{2k+2}\geq\dots\geq N_{4k+2}}\Big|\int_{0}^{\delta}\int_{\sum_{i=1}^{4k+2}\xi_{i}=0}\Big(1-\frac{m_{2k+2\dots 4k+2}}{m_{2k+2}\dots m_{4k+2}}\Big)\reallywidehat{P_{N_{1\dots2k+1}}\overline{I(\vert u\vert^{2k}u)}}(\xi_{1\dots2k+1})\\\times\widehat{Iu_{N_{2}}}(\xi_{2})\dots\widehat{Iu_{N_{4k+2}}}(\xi_{4k+2})\Big|
\end{multline*}
where $P_{N_{1\dots2k+1}}$ is the Littlewood-Paley projection operator onto the dyadic shell $N_{1\dots2k+1}\sim\langle\xi_{1\dots2k+1}\rangle$. Assuming the spatial Fourier transform of dyadic pieces is nonnegative, we take the multiplier out of the integral with the pointwise bound
\begin{equation*}
    \Big|1-\frac{m_{2k+2\dots4k+2}}{m_{2k+2}\dots m_{4k+2}}\Big|\lesssim\frac{m(N_{1\dots2k+1})}{m(N_{2k+2})\dots m(N_{4k+2})}
\end{equation*}
for each frequency interaction case. As before, by symmetry, we may assume that $N_{2k+2}\geq\dots\geq N_{4k+2}$. Then we can infer that $N_{1\dots2k+1}\lesssim N_{2k+2}$. So we need to show that
\begin{multline}\label{eq:term3_dyad_p_wish}
    \Big|\int_{0}^{\delta}\int_{\sum_{i=1}^{4k+2}\xi_{i}=0}\Big(1-\frac{m_{2k+2\dots 4k+2}}{m_{2k+2}\dots m_{4k+2}}\Big)\reallywidehat{P_{N_{1\dots2k+1}}\overline{I(\vert u\vert^{2k}u)}}(\xi_{1\dots2k+1})\\\times\widehat{Iu_{N_{2}}}(\xi_{2})\dots\widehat{Iu_{N_{4k+2}}}(\xi_{4k+2})\Big|\lesssim N_{2k+2}^{0-}N^{-4+}\Vert Iu\Vert_{X^{2,\frac{1}{2}+}_{\delta}}^{2k+1}\prod_{j=2k+2}^{4k+2}\Vert Iu_{N_{j}}\Vert_{X^{2,\frac{1}{2}+}_{\delta}}.
\end{multline}
The factor $N^{0-}_{2k+2}$ on the right side of the inequality allows us to sum all the dyadic pieces over the dyadic numbers $N_{1\dots2k+1},\dots,N_{4k+2}$ and we do not have to deal with the dyadic pieces of separate terms of $I(\vert u\vert^{2k}u)$. Therefore, using the bound on the multiplier and utilizing Hölder's inequality after the Fourier inversion, we arrive at
\begin{multline}\label{eq:term3_dyad_last}
    \text{LHS of \eqref{eq:term3_dyad_p_wish}}\lesssim\frac{m(N_{1\dots2k+1})}{m(N_{2k+2})\dots m(N_{4k+2})}\Vert P_{N_{1\dots2k+1}}I(\vert u\vert^{2k}u)\Vert_{L^{2}_{t,x}}\Vert Iu_{N_{2k+2}}\Vert_{L^{6}_{t,x}}\Vert Iu_{N_{2k+3}}\Vert_{L^{6}_{t,x}}\\\times\Vert Iu_{N_{2k+4}}\Vert_{L^{6}_{t,x}}\prod_{j=2k+5}^{4k+2}\Vert u_{N_{j}}\Vert_{L^{\infty}_{t,x}}
\end{multline}
To continue with estimating RHS of \eqref{eq:term3_dyad_last}, we require the following estimates:
\begin{lemma}\label{yardımcı}
    Let $u_{N_{2k+2}},\dots,u_{N_{4k+2}}$ be defined as above. Then
    \begin{align}
        \Vert P_{N_{1\dots2k+1}}I(\vert u\vert^{2k}u)\Vert_{L^{2}_{t,x}}&\lesssim\langle N_{1\dots2k+1}\rangle^{-2}\Vert Iu\Vert_{X^{2,\frac{1}{2}+}_{\delta}}^{2k+1}\label{eq:eq1},\\
        \Vert Iu_{N_{j}}\Vert_{L^{6}_{t,x}}&\lesssim\langle N_{j}\rangle^{-2}\Vert Iu\Vert_{X^{2,\frac{1}{2}+}_{\delta}},\quad j=2k+2,2k+3,2k+4,\label{eq:eq2}\\
        \Vert Iu_{N_{j}}\Vert_{L^{\infty}_{t,x}}&\lesssim\langle N_{j}\rangle^{-1}\Vert Iu_{N_{j}}\Vert_{X^{2,\frac{1}{2}+}_{\delta}},\quad j=2k+5,\dots,4k+2.\label{eq:eq3}
    \end{align}
\end{lemma}
\begin{proof}
    To show \eqref{eq:eq1}, by Bernstein's inequality \eqref{eq:Berns_ineq} and $L^{2}\to L^{2}$ boundedness of $P_{N_{1\dots2k+1}}$, it suffices to demonstrate that
    \begin{equation*}
         \Vert\langle\nabla\rangle^{2} P_{N_{1\dots2k+1}}I(\vert u\vert^{2k}u)\Vert_{L^{2}_{t,x}}\lesssim\Vert Iu\Vert_{X^{2,\frac{1}{2}+}_{\delta}}^{2k+1}.
    \end{equation*}
    The pseudo-differential operator $\langle\nabla\rangle^{2}I$ is of positive order $s$. So, it obeys the fractional Leibniz rule. Applying Hölder's inequality, Sobolev embedding and \eqref{keyestimate} with $\mu=0$, we get
    \begin{align*}
         \Vert\langle\nabla\rangle^{2} P_{N_{1\dots2k+1}}I(\vert u\vert^{2k}u)\Vert_{L^{2}_{t,x}}\lesssim\Vert\langle\nabla\rangle^{2}Iu\Vert_{L^{6}_{t,x}}\Vert Iu\Vert_{L^{6k}_{t,x}}^{2k}&\lesssim\Vert Iu\Vert_{X^{2,\frac{1}{2}+}_{\delta}}\Vert\langle\nabla\rangle^{\frac{k-1}{k}}u\Vert_{L^{6k}_{t}L^{\frac{6k}{3k-2}}_{x}}^{2k}\\
         &\lesssim\Vert Iu\Vert_{X^{2,\frac{1}{2}+}_{\delta}}\Vert\langle\nabla\rangle^{\frac{k-1}{k}} u\Vert_{X^{0,\frac{1}{2}+}_{\delta}}^{2k}\\
         &\lesssim\Vert Iu\Vert_{X^{2,\frac{1}{2}+}_{\delta}}^{2k+1}
    \end{align*}
    provided that $s\geq\frac{k-1}{k}$. \eqref{eq:eq2} follows from the inequality \eqref{keyestimate} with $\mu=0$ and Bernstein's inequality \eqref{eq:Berns_ineq}
    \begin{equation*}
         \Vert Iu_{N_{j}}\Vert_{L^{6}_{t,x}}\lesssim \Vert Iu_{N_{j}}\Vert_{X^{0,\frac{1}{2}+}_{\delta}}\sim\langle N_{j}\rangle^{-2} \Vert Iu_{N_{j}}\Vert_{X^{2,\frac{1}{2}+}_{\delta}},\quad j=2k+2,2k+3,2k+4.
    \end{equation*}
    To establish \eqref{eq:eq3}, we utilize the Fourier inversion formula and Plancherel's theorem:
    \begin{align*}
        \vert Iu_{N_{j}}(x)&\vert\leq\int_{\vert\xi\vert\sim N_{j}}\vert \widehat{Iu_{N_{j}}}(\xi)\vert d\xi\\&
        \leq\Big(\int_{\vert\xi\vert\sim N_{j}}\langle\xi\rangle^{-4}d\xi\Big)^{1/2}\Big(\int_{\vert\xi\vert\sim N_{j}}\langle\xi\rangle^{4}\vert\widehat{Iu_{N_{j}}}(\xi)\vert^{2} d\xi\Big)^{1/2}\\
        &\lesssim \langle N_{j}\rangle^{-1}\Vert Iu_{N_{j}}\Vert_{H^{2}_{x}}.
    \end{align*}
    Taking the supremum over $x$ and using the embedding $X^{2,\frac{1}{2}+}_{\delta}\hookrightarrow C^{0}_{t}H^{2}_{x}$ lead to the desired result.
\end{proof}
\noindent Using the above Lemma \ref{yardımcı}, we get
\begin{multline*}
    \text{RHS of \eqref{eq:term3_dyad_last}}\lesssim\frac{m(N_{1\dots2k+1})}{m(N_{2k+2})\dots m(N_{4k+2})}\langle N_{2k+2}\rangle^{-2}\langle N_{2k+3}\rangle^{-2}\langle N_{2k+4}\rangle^{-2}\prod_{j=2k+5}^{4k+2}\langle N_{j}\rangle^{-1}\\\times\Vert Iu\Vert_{X^{2,\frac{1}{2}+}_{\delta}}^{2k+1}\prod_{j=2k+2}^{4k+2}\Vert Iu_{N_{j}}\Vert_{X^{2,\frac{1}{2}+}_{\delta}}.
\end{multline*}
Thus, it suffices to exhibit that
\begin{equation}\label{eq:term3_wish_to_show}
    \frac{m(N_{1\dots2k+1})}{m(N_{2k+2})\dots m(N_{4k+2})}\frac{\prod_{j=2k+5}^{4k+2}N_{j}^{-1}}{\langle N_{1\dots2k+1}\rangle^{2}N_{2k+2}^{2}N_{2k+3}^{2}N_{2k+4}^{2}}\lesssim N^{-4+}N_{2k+2}^{0-}
\end{equation}
for each frequency interaction case.\\
\noindent
\textbf{Case 1: }$N_{1\dots2k+1}\sim N_{2k+2}\gtrsim N\gg N_{2k+3}$.\\
\noindent
In this case, we have
\begin{equation*}
    \text{LHS of \eqref{eq:term3_wish_to_show}}\lesssim N_{2k+2}^{-4}.
\end{equation*}
As $N_{2k+2}\gtrsim N$, \eqref{eq:term3_wish_to_show} holds at once.\\
\noindent
\textbf{Case 2: }$N_{2k+2}\sim N_{2k+3}\gtrsim N$.\\
\noindent
In this case, we have
\begin{align*}
\text{LHS of \eqref{eq:term3_wish_to_show}}&\lesssim\frac{m(N_{1\dots2k+1})}{m(N_{2k+2})^{2}N_{2k+2}^{4}\langle N_{1\dots2k+1}\rangle^{2}m(N_{2k+4})N_{2k+4}^{2}\prod_{j=2k+5}^{4k+2}m(N_{j})N_{j}}\\
&\lesssim\frac{1}{m(N_{2k+2})^{2}N_{2k+2}^{4}}\\
&\lesssim\frac{N_{2k+2}^{0-}}{N^{4-}}
\end{align*}
where we have used the facts that $m(N_{1\dots2k+1})\langle N_{1\dots2k+1}\rangle^{-2}\lesssim 1$ and $m(x)x^{p}$, $x\geq 0$, is increasing provided that $p>\frac{3}{4k}$. 
\\
\noindent
\textbf{Case 3: }$N_{1\dots2k+1}\sim N_{2k+2}\gg N_{2k+3}\gtrsim N$.\\
\noindent
In this case
\begin{align*}
    \text{LHS of \eqref{eq:term3_wish_to_show}}&\lesssim\frac{m(N_{1\dots2k+1})}{m(N_{2k+2})\dots m(N_{4k+2})}\frac{\prod_{j=2k+5}^{4k+2}N_{j}^{-1}}{\langle N_{1\dots2k+1}\rangle^{2}N_{2k+2}^{2}N_{2k+3}^{2}N_{2k+4}^{2}}\\
    &\lesssim\frac{1}{N_{2k+2}^{4}m(N_{2k+3})N_{2k+3}^{2}\prod_{j=2k+5}^{4k+2}N_{j}m(N_{j})}\\
    &\lesssim N^{-2}N_{2k+2}^{-4}
\end{align*}
where we have exploited the facts that $m(N_{2k+3})N_{2k+3}^{2}\gtrsim N^{2}$, $m(N_{j})N_{j}\gtrsim 1$, $j=2k+5,\dots,4k+2$. Therefore, as $N_{2k+2}\gtrsim N$
\begin{equation}
    N^{-2}N_{2k+2}^{-4}\lesssim N^{-4+}N_{2k+2}^{0-}.
\end{equation}
Consequently, the proof of the growth of almost conserved quantity is completed. 
\end{proof}
\section{Proof of the Main Theorem}
In this section, we shall prove Theorem \ref{eq:main_thm} by making use of Proposition \ref{almost_cons_law}. By almost conservation law and the modified local existence, the growth of the modified energy can be bounded by
\begin{align}
    E(Iu(\delta))-E(Iu(0))&\lesssim N^{-3+}\Vert Iu\Vert_{X^{2,\frac{1}{2}+}_{\delta}}^{2k+2}+N^{-4+}\Vert Iu\Vert_{X^{2,\frac{1}{2}+}_{\delta}}^{4k+2}\nonumber\\
    &\lesssim N^{-3+}\Vert Iu_{0}\Vert_{H^{2}_{x}}^{2k+2}+N^{-4+}\Vert Iu_{0}\Vert_{H^{2}_{x}}^{4k+2}\nonumber\\
    &\lesssim N^{-3+}N^{(2k+2)(2-s)}\Vert u_{0}\Vert_{H^{s}_{x}}^{2k+2}+N^{-4+}N^{(4k+2)(2-s)}\Vert u_{0}\Vert_{H^{s}_{x}}^{4k+2}\nonumber\\
    &\lesssim N^{-3+}N^{(2k+2)(2-s)}+N^{-4+}N^{(4k+2)(2-s)}\label{process}.
\end{align}
Moreover, to determine the initial modified energy in terms of the frequency cut-off parameter $N$, we need the following estimate for the $L^{2k+2}$-norm piece of the almost conserved quantity.
\begin{lemma}
    If $u_{0}\in H^{s}_{x}(\mathbb{R}^{2})$ for $2>s\geq\frac{k}{k+1}$, then
    \begin{equation}\label{est.}
    \norm{Iu_0}_{L^{2k+2}_x}\leq CN^\frac{2-s}{k+1}\norm{u_0}_{H^s_x}.
    \end{equation}
\end{lemma}
\begin{proof}
 We set
 \begin{align*}
 u_0=P_{\leq N}u_0 + P_{> N}u_0 .    
 \end{align*}Showing \eqref{est.} for the 
 high-frequency part of initial data suffices, as for the low-frequency part the $I$-operator acts as identity. Indeed, the Sobolev embedding $H^s(\mathbb{R}^2)\hookrightarrow L^{2k+2}(\mathbb{R}^2)$ for $s\geq \frac{k}{k+1}$ implies that
 \begin{align}\label{first}
     \norm{IP_{\leq N}u_0}_{L^{2k+2}}\leq C\norm{u_0}_{H^s} 
 \end{align}
for some positive real constant $C$. For the high frequency part, again by the embedding $H^{\frac{k}{k+1}}(\mathbb{R}^2)\hookrightarrow L^{2k+2}(\mathbb{R}^2)$ and Bernstein's inequality \eqref{eq:Berns_ineq2} together with the fact that $I$-operator is smoothing of order $2-s$, we obtain 
 \begin{align}
  \norm{IP_{> N}u_0}_{L^{2k+2}}&\lesssim \norm{|\nabla|^{\frac{k}{k+1}}IP_{\geq N}u_0}_{L^2(\mathbb{R}^2)}  \lesssim N^{\frac{k}{k+1}-2}\norm{|\nabla|^{2}IP_{\geq N}u_0}_{L^2(\mathbb{R}^2)}\nonumber
  \\&\lesssim N^{\frac{k}{k+1}-2}\norm{IP_{\geq N}u_0}_{H^2(\mathbb{R}^2)}\nonumber\lesssim N^{\frac{k}{k+1}-2}N^{2-s}\norm{P_{\geq N}u_0}_{H^s(\mathbb{R}^2)}\nonumber
  \\&\lesssim N^{\frac{k}{k+1}-s}\norm{u_0}_{H^s(\mathbb{R}^2)}\leq  N^{\frac{2-s}{k+1}}\norm{u_0}_{H^s(\mathbb{R}^2)} \label{second}
 \end{align}
 where the last inequality is valid for the exponents for $s\geq\frac{k-2}{k}$ but this range is already subsumed by the assumption. Thus the inequality \eqref{est.} follows from \eqref{first} and \eqref{second} since $N\gg1$ and $s<2$.
\end{proof}
\noindent Since $u_0\in H^s$ for $s\geq1$, using \eqref{est.} along with the fact that $I$-operator is smoothing of order $2-s$, for $1\leq s<2$, we arrive at:
\begin{align*}
    E(Iu(0))&\leq C_0\Vert Iu_{0}\Vert_{H^{2}_{x}}^{2}+C_0\Vert Iu_{0}\Vert_{H^{1}_{x}}^{2}+C_0\Vert Iu_{0}\Vert_{L^{2k+2}_{x}}^{2k+2}\\
    &\leq C_1N^{2(2-s)}\Vert u_{0}\Vert_{H^{s}_{x}}^{2}+C_1N^{2(2-s)}\norm{u_0}_{H^1_x}^2+C_1N^{2(2-s)}\Vert u_{0}\Vert_{H^{s}_{x}}^{2k+2}\\
    &\leq CN^{2(2-s)}\big(\Vert u_{0}\Vert_{H^{s}_{x}}^{2}+\Vert u_{0}\Vert_{H^{s}_{x}}^{2k+2}\big)\lesssim N^{2(2-s)}.
\end{align*}
Let $T\gg 1$ be a large time parameter. From Lemma \ref{mod_local_exst} along with the smoothing behavior of $I$-operator, we set $\delta^{1-}\sim N^{-2k(2-s)}$. If we iterate the process \eqref{process} $T\delta^{-1}$ times until we reach $E(Iu(T))\sim N^{2(2-s)}$, then we require that the following relation should hold
\begin{equation}\label{eq:T_delta_N_relation}
    T\delta^{-1}\Big(N^{-3+}N^{(2k+2)(2-s)}+N^{-4+}N^{(4k+2)(2-s)}\Big)\sim N^{2(2-s)}.
\end{equation}
This occurs provided that we have
\begin{equation*}
    T\sim N^{4ks-8k+3-}+N^{6ks-12k+4-}.
\end{equation*}
The exponents of $N$ are positive only when
\begin{itemize}
    \item $4ks-8k+3->0\Rightarrow s>2-\frac{3}{4k},$
    \item $6ks-12k+4->0\Rightarrow s>2-\frac{2}{3k}.$
\end{itemize}
 Finally, we need to establish the polynomial-in-time bound for the $H^{s}_{x}$-norm of $u$. By the relation between $N$ and $T$, for $T\gg 1$, we have
\begin{equation}
    \sup_{t\in[0,T]}\Vert u(t)\Vert_{H^{s}_{x}}\lesssim E(I(u(T))\lesssim N^{2(2-s)}\sim
    \begin{cases}  
     T^{\frac{4-2s}{4ks-8k+3}-}&\text{if $2-\frac{3}{4k}<s\leq 2-\frac{1}{2k}$}\\
     T^{\frac{4-2s}{6ks-12k+4}-}&\text{if $2-\frac{1}{2k}\leq s<2$}.
     \end{cases}
\end{equation}

\section*{Acknowledgement}
The third author would like to thank his Ph.D advisor T. Burak G\"{u}rel for many helpful suggestions and comments during the preparation of this manuscript.

\nocite{*}
\bibliographystyle{abbrv}
\bibliography{reference.bib}

\end{document}